\documentclass{article}
\usepackage{amsfonts,amsmath,amssymb,amsthm,enumitem,nccmath}

\PassOptionsToPackage{obeyspaces}{url}
\usepackage[colorlinks=true,citecolor=blue,urlcolor=blue,linkcolor=blue,bookmarksopen=true]{hyperref}
\usepackage{breakurl}
\usepackage{tikz}
\usetikzlibrary{calc}

\usepackage{etoolbox} 
\patchcmd{\thebibliography}{\leftmargin\labelwidth}{\leftmargin\labelwidth\addtolength\itemsep{-0.1\baselineskip}}{}{}

\author{Zhiyang He\thanks{Department of Mathematical Sciences, Carnegie Mellon University, Pittsburgh, PA, 15213, USA. This research was supported by a Summer Undergraduate Research Fellowship at CMU. \texttt{szh@andrew.cmu.edu}} \and Michael Tait\thanks{Department of Mathematical Sciences, Carnegie Mellon University, Pittsburgh, PA, 15213, USA. Supported in part by NSF grant DMS-1606350. \texttt{mtait@cmu.edu}}}

\title{Hypergraphs with few Berge paths of fixed length between vertices}
\date{}
\newtheoremstyle{mystyle}
  {}
  {}
  {\itshape}
  {}
  {\bfseries}
  {.}
  { }
  {\thmname{#1}\thmnumber{ #2}\thmnote{ (#3)}}%

\theoremstyle{mystyle}

\newtheorem{theorem}{Theorem}
\newtheorem{lemma}{Lemma}
\newtheorem{claim}{Claim}

\DeclareMathOperator{\ex}{ex}

\newcommand*{\TODO}[1]{\marginpar{\usefont{T1}{fxb}{o}{b}\fontsize{24pt}{26pt}\selectfont!}}

\begin{document}
\maketitle

\begin{abstract}
In this paper we study the maximum number of hyperedges which may be in an $r$-uniform hypergraph under the restriction that no pair of vertices has more than $t$ Berge paths of length $k$ between them. When $r=t=2$, this is the even-cycle problem asking for $\mathrm{ex}(n, C_{2k})$. We extend results of F\"uredi and Simonovits and of Conlon, who studied the problem when $r=2$. In particular, we show that for fixed $k$ and $r$, there is a constant $t$ such that the maximum number of edges can be determined in order of magnitude. 
\end{abstract}

\section{Introduction}

The {\em Tur\'an number} for a fixed graph $F$ is the maximum number of edges that an $n$ vertex graph may have without containing $F$ as a subgraph. This function is denoted by $\mathrm{ex}(n, F)$ and its study for various graphs $F$ is a classical problem in extremal combinatorics. The celebrated Erd\H{o}s-Stone theorem \cite{ErdosStone} gives that 
\[
\mathrm{ex}(n, F) = \left(1 - \frac{1}{\chi(F)-1} + o(1)\right) \binom{n}{2},
\]
and asymptotically solves the problem whenever $\chi(F) \geq 3$. However, for bipartite $F$ it only gives $\mathrm{ex}(n, F) = o(n^2)$, and determining the order of magnitude for the Tur\'an number of a bipartite graph in general is a notoriously difficult problem (see \cite{FurediSimonovits} for a survey). 

When $F = C_{2k}$, the study of $\mathrm{ex}(n, C_{2k})$ is known as the even-cycle problem. This problem was first studied in 1938 by Erd\H{o}s \cite{Erdos1938} and has since become a central problem in extremal graph theory. A general upper bound of $\mathrm{ex}(n, C_{2k}) = O_k(n^{1+1/k})$ was first published by Bondy and Simonovits \cite{BondySimonovits}. Since then, improvements have been made to the multiplicative constant \cite{Verstraete, Pikhurko, BukhJiang}, and constructions have been found showing that the order of magnitude is correct for $k\in \{2,3,5\}$ \cite{Brown, ErdosRenyiSos, Wenger, Benson}. However, besides $C_4$, $C_6$, and $C_{10}$, the order of magnitude is unknown. 

We are interested in a generalization of this problem to hypergraphs. Note that a $C_{2k}$ is a pair of internally disjoint paths of length $k$ between a fixed pair of vertices. We will study hypergraphs where we forbid a certain number of paths of length $k$ between vertices. In a hypergraph $H$, a {\em Berge path} of length $k$ is a set of distinct vertices $v_0,v_1,\cdots, v_k$ and a set of distinct hyperedges $h_1,h_2,\cdots, h_k$ such that $\{v_{i-1}, v_i\} \in h_i$ for $1\leq i\leq k$. Note that the hyperedges $\{h_i\}$ could intersect in many different ways and in general two Berge paths of length $k$ need not be isomorphic. We call the vertices $v_0,\cdots, v_k$ the {\em core vertices} of the Berge path.

The definition of Berge paths and cycles was extended to arbitrary graphs by Gerbner and Palmer \cite{GerbnerPalmer}. We say a hypergraph $H$ is a Berge-$F$ if there is a bijection $\phi:E(F) \to E(H)$ such that $e \subset \phi(e)$ for all $e\in E(F)$. Again note that two hypergraphs $H_1$ and $H_2$ may each be a Berge-$F$ and could be non-isomorphic. We denote by $F^{B, r}$ the set of all $r$-uniform hypergraphs which are a Berge-$F$ and we will write $F^B$ when the uniformity is fixed. Given a family of $r$-uniform hypergraphs $\mathcal{F}$ we denote the maximum number of hyperedges in an $n$ vertex $r$-uniform hypergraph which does not contain any $F\in \mathcal{F}$ as a subhypergraph by $\mathrm{ex}_r(n, \mathcal{F})$. 

In this paper we are interested in studying $\mathrm{ex}_r(n, \mathcal{F})$ when $\mathcal{F}$ is a family of Berge theta graphs. A {\em theta graph}, denoted by $\Theta_{k,t}$ is the ($2$-uniform) graph given by a set of $t$ internally disjoint paths of length $k$ between a fixed pair of vertices. Note that $\Theta_{k,2} = C_{2k}$ and so the study of the Tur\'an number of $\Theta_{k,t}$ generalizes the even-cycle problem. We will study $r$-uniform hypergraphs which do not contain a Berge-$\Theta_{k,t}$, i.e. we will study the Tur\'an number for the family $\Theta_{k,t}^B$. An alternative definition of a hypergraph in $\Theta_{k,t}^B$ is a set of distinct vertices $x,y, v_1^1,\cdots, v_{k-1}^1, \cdots, v_{1}^t,\cdots, v_{k-1}^t$ and a set of distinct $r$-edges $h_1^1,\cdots, h_k^1,\cdots , h_1^t, \cdots, h_k^t$ such that $\{x, v_1^i\} \subset h_1^i$, $\{v_{j-1}^i, v_{j}^i\} \subset h_{j}^i$, and $\{v_k^i, y\} \subset h_k^i$ for $1\leq i\leq t$ and $2\leq j\leq k-1$. That is, we are forbidding that a pair of vertices have $t$ Berge paths of length $k$ between them with disjoint internal core vertices. For ($2$-uniform) graphs, the study of $\mathrm{ex}(n, \Theta_{k,t})$ has been investigated in \cite{BukhTait, Conlon, FaudreeSimonovits}. In particular, Faudree and Simonovits gave a more general version of the upper bound in the even-cycle problem. 

\begin{theorem}[Faudree and Simonovits \cite{FaudreeSimonovits}]\label{FS}
Given two integers $k$ and $t$, there exists a constant $c_{k,t} > 0$, such that 
$$\ex(n, \Theta_{k,t}) \le c_{k,t}n^{1+\frac{1}{k}}.$$
\end{theorem}

Our first main result is an extension of this bound to $r$-uniform hypergraphs. 

\begin{theorem}\label{upperbound}
For fixed $r$, given two integers $k$ and $t$, there exists a constant $c_{r,k,t}$, such that 
$$\ex_r(n, \Theta_{k,t}^B) \le c_{r,k,t}n^{1+\frac{1}{k}}.$$
\end{theorem}

Recently, Conlon \cite{Conlon} complemented the upper bound of Faudree and Simonovits and showed that the order of magnitude is correct when $t$ is large enough relative to $k$.

\begin{theorem}[Conlon \cite{Conlon}]\label{Conlon}
For any natural number $k\ge 2$, there exists a natural number $t$ such that 
$$\ex(n, \Theta_{k,t}) = \Omega_k(n^{1+\frac{1}{k}}).$$
\end{theorem}

Our second theorem shows that this is also the case in higher uniformities.

\begin{theorem}\label{lowerbound}
For fixed $r$ and any natural number $k\ge 2$, there exists a natural number $t$ such that 
$$\ex_r(n, \Theta_{k,t}^B) = \Omega_{k,r}(n^{1+\frac{1}{k}}).$$
\end{theorem}

In Section \ref{section:upper bound} we prove Theorem \ref{upperbound} and in Section \ref{lower bound proof} we prove Theorem \ref{lowerbound}

\section{Proof of Theorem \ref{upperbound}}\label{section:upper bound}

The proof of this theorem is inspired by a reduction lemma of Gy\H{o}ri and Lemons \cite{GL}. To prove Theorem~\ref{upperbound} from Theorem~\ref{FS}, we prove the following lemma, which says that given any $r$-uniform $\Theta_{k,t}^B$-free hypergraph, we can reduce it to a $\Theta_{k,t}$-free ($2$-uniform) graph which has a constant proportion of the edges of the original $r$-uniform graph. 

\begin{lemma}[Reduction Lemma]
Let $2\le m< r$, and $H$ be a $r$-uniform hypergraph without $\Theta_{k,t}^B$. Define the $m$-uniform hypergraph $G$ in the following way: Order the edges of $H$ arbitrarily and, going through the edges one by one, pick an $m$-set from each hyperedge to be in $G$ where the $m$ set chosen is the one which has been chosen the fewest number of times previously (break ties arbitrarily). Let $M_{k,i,r,m} = \sum_{j=1}^{k+1}{mk(i-1)+jm-m\choose r-m}+k+1$. Then the hypergraph $G$ will have no $\Theta_{k,t}^{B,m}$ and each of its edges will have multiplicity no more than $M_{k,t,r,m}$.
\end{lemma}

\begin{proof}
Let's first fix the notation we'll use in the rest of the proof. We will call the edges of $H$ as hyperedges and the edges of $G$ just as edges. We'll use $h_i$ to denote the hyperedges, and $e_i$ to denote edges. The edge set of $H$ and $G$ will be $E(H)$ and $E(G)$, respectively, and their (common) vertex set will be denoted as $V$. For every $h\in E(H)$, the $m$-set chosen from $h$ will be denoted as $e(h)$. 

First note that $G$ is $\Theta_{k,t}^{B,m}$ free. Indeed, for any Berge path in $G$, say $v_0e_1v_1e_2...e_{k}v_k$, there is a corresponding Berge path in $H$ $v_0h_1...h_{k}v_k$ where $h_i$ is the hyperedge such that $e_i = e(h_i)$. Therefore, a $\Theta_{k,t}^{B,m}$ in $G$ would imply that there is a $\Theta_{k,t}^{B,r}$ in $H$.

Now assume the lemma is not true and that there is an edge in $G$ which was chosen more than $M_{k,t,r,m}$ times. It suffices for us to construct a $\Theta_{k,t}^{B,m}$ in $G$, which will give us a $\Theta_{k,t}^{B,r}$ in $H$ and results in a contradiction. Let $e$ be an edge in $G$ with some multiplicity at least $M+1$, and $x$, $y\in e$. Consider the last hyperedge which contributed to the multiplicity of $e$, call it $h$. Since every time we select an $m$-set from a hyperedge in $H$, we select the one with the least multiplicity, we know that every other $m$-set in $h$ must have multiplicity at least $M$. Since $r > 2$, we know there exists $v_1\in h\setminus \{x, y\}$. Therefore there exists a hyperedge $h_1\in E(H)$ such that $\{x, v_1\}\subseteq e(h_1)$. On the other hand, note that $v_1, y\in h\setminus \{x\}$, thus there exists $h'_1\in E(H) \setminus \{h_1\}$ such that $\{v_1, y\}\subseteq e(h'_1)$. This gives us a length 2 path from $x$ to $y$, namely $x, v_1, y$ connected by $h_1$ and $h'_1$. Note that $e(h_1)$ and $e(h'_1)$ also have multiplicity at least $M$, and $h$ is not part of the path. 

To extend this result, we prove the following claim, from which the construction follows easily.

\begin{claim}
Let $S\subset V$ be a ``forbidden set", and $M = \sum_{j=1}^{k-1}{|S|+jm-m\choose r-m}+k+1$. For $i \le k-2$, Suppose we have vertices $x, v_1, v_2\cdots v_i, y$ and edges $e_1, \cdots, e_{i+1}$ that forms a path in $G$ in the order given, and the last edge $e_{i+1}\in G$ with $v_i, y\in e_{i+1}$ has multiplicity at least $M - \sum_{j=1}^{i}{|S|+jm-m\choose r-m}-(i-1)+1$. Let $S' = S \cup e_1 \cup\cdots\cup e_{i+1}$. Then we can find $v_{i+1} \notin S'$ such that $x, v_1, v_2\cdots v_i, v_{i+1}, y$ forms a path, and the last edge in this new path containing $v_{i+1}, y$ has multiplicity at least $M - \sum_{j=1}^{i+1}{|S|+jm-m\choose r-m}-i+1$.
\end{claim} 

\begin{proof}
To find the vertex $v_{i+1}$, first we will find a new hyperedge $h'\in E(H)$ such that $e_{i+1}\subset h'$ and $h'\not\subseteq S'$. Note that the set $S'$ has cardinality less than $|S|+(i+1)m$ as we included all $i+1$ edges in the already existing path. We would like to make sure the next edge we choose for the new path is different from all previous edges. Since $e_{i+1}$ must be in $h'$, the number of hyperedges $h$ such that $e(h) = e_{i+1}$ and $h\subset S'$ is at most ${|S|+(i+1)m-m\choose r-m}$. Now we let $E' = \{h\in E(H)\mid e(h) = e_{i+1}, h\not\subseteq S'\}$, then we have
$$
|E'| \ge |\{h\in E(H)\mid e(h) = e_i\}| - {|S|+(i+1)m-m\choose r-m} \ge M - \sum_{j=1}^{i+1}{|S|+jm-m\choose r-m}-(i-1)+1.
$$
Now we pick $h'$ to be the last edge (in the original ordering) of $E'$, which then implies that every $m$-set in $h'$ besides $e_{i+1}$ has multiplicity at least $M - \sum_{j=1}^{i+1}{|S|+jm-m\choose r-m}-i+1$. Since $h'\not\subseteq S'$, we can find $v_{i+1}\in h'$ such that $v_{i+1}\notin S'$. Now we choose two $m$-sets from $h'$, namely $e'_{i+1}$ and $e_{i+2}$ where $v_i, v_{i+1}\in e'_{i+1}$ and $v_{i+1}, y\in e_{i+2}$. This gives us a length $i+2$ path from $x$ to $y$, namely the path $x, v_1, \cdots v_i, v_{i+1}, y$ where the last two edges are $e'_{i+1}$ and $e_{i+2}$, and the previous edges are the same as in the old length $i+1$ path. The last edge, $e_{i+2}$, has multiplicity at least $M - \sum_{j=1}^{i+1}{|S|+jm-m\choose r-m}-i+1$, as desired. Note that the hyperedge $h'$ is not part of the path, and as we discard the edge $e_{i+1}$ in the original path, we add in two new edges, namely $e'_{i+1}$ and $e_{i+2}$.
\end{proof}

With this claim, we can now build the first path from $x$ to $y$ by induction on $k$. The base case when $k=2$ is already constructed before the statement of the claim. In the induction step, when we already has a length $k-1$ path, apply the claim with $S = \varnothing$ and we obtain a length $k$ path, as desired. Note that if all we need is just one path, then we just need the multiplicity of the edge $e$ to be at least $M_{k,1,r,m}$. 

To build $t$ paths, we will construct each path separately. Assume that we have built $i-1$ paths. To build the $i$th path, the forbidden set $S$ would be the union of all edges in previous paths. This ensures that the paths we are building are vertex-independent and edge-distinct. It then follows that to build $t$ vertex-independent paths from $x$ to $y$, it suffices for us to have that the multiplicity of the edge $e$ to be at least $M_{k,t,r,m}$. This gives us a $\Theta_{k,t}^{B,m}$ in the graph $G$. If we then choose hyperedges in $H$ that gives the edges in this $\Theta_{k,t}^{B,m}$, we obtain a $\Theta_{k,t}^{B,r}$ in $H$. Such hyperedges can be chosen because all edges in this $\Theta_{k,t}^{B,m}$ in $G$ have multiplicity at least 1, and these hyperedges are guaranteed to be distinct since for any two hyperedges chosen, say $h_1$ and $h_2$, $e(h_1)\ne e(h_2)$, and therefore $h_1\ne h_2$. This leads to a contradiction. Therefore all edges in $G$ must have multiplicity at most $M_{k,t,r,m}$.
\end{proof}

By Theorem~\ref{FS}, there exists a constant $c_{k,t}$ such that $\ex(n, \Theta_{k,t}) \le c_{k,t}n^{1+\frac{1}{k}}$. Now assume that a $r$-uniform hypergraph has more than $M_{k,t,r,2}c_{k,t}n^{1+\frac{1}{k}}$ edges. Then if we reduce this hypergraph into a ($2$-uniform) graph with the scheme described above, we will either have a graph with more than $c_{k,t}n^{1+\frac{1}{k}}$ edges, or it will have at least one edge with more than $M_{k,t,r,2}$ multiplicity. Both cases imply that there is a $\Theta_{k,t}$ in the reduced graph, which leads to a $\Theta_{k,t}^B$ in the original hypergraph. This completes the proof of Theorem~\ref{upperbound}.

\section{Preliminaries for Theorem \ref{lowerbound}}

To prove Theorem~\ref{lowerbound}, given any natural number $k\le 2$, we need to construct a $\Theta_{k,t}^B$-free $r$-uniform hypergraph with $\Omega(n^{1+\frac{1}{k}})$ edges where $t$ is a large enough constant depending only on $k$ and $r$. Blagojevi\'c, Bukh and Karasev \cite{BKP} found an elegant random algebraic construction for Tur\'an type problems, and we will use this method to construct our hypergraphs. This method has recently been used with success in both the graph \cite{Bukh, BukhTait, Conlon} and hypergraph setting \cite{Jie}.

Let $k$ and $r$ be fixed, and for $q$ a prime power let $\mathbb{F}_q$ be the finite field of order $q$. We will work with polynomials in $kr$ variables over $\mathbb{F}_q$. Let $P_d$ be the set of such polynomials of degree at most $d$. That is, $P_d$ consists of linear combinations of monomials $\prod_{i=1}^{kr} x_i^{\alpha_i}$ where $\sum a_i \leq d$. For the remainder of the paper we will use the term {\em random polynomial} to denote a polynomial chosen uniformly at random from $P_d$. Note that the distribution of random polynomials is equivalent to choosing the coefficient of each monomial $\prod_{i=1}^{kr} x_i^{\alpha_i}$ independently and uniformly from $\mathbb{F}_q$. 


We will need to know the probability that a random polynomial vanishes on a fixed set of points. In particular, if we fix one point, since the constant term of a random polynomial is chosen uniformly from $\mathbb{F}_q$, we have the following lemma.

\begin{lemma}\label{Ind1}
If $f$ is a random polynomial from $P_d$, then, for any fixed $x\in F_{q}^{t}$,
$$
\mathbb{P}[f(x) = 0] = \frac{1}{q}.
$$
\end{lemma}

The next lemma, proved as Lemma 2.3 in \cite{BukhConlon} and Lemma 2 in \cite{Conlon}  extends the conclusion of Lemma~\ref{Ind1}.

\begin{lemma}\label{Ind2}
Assume $x_1, \cdots, x_z$ are $z$ distinct points in $\mathbb{F}_q^{t}$. Suppose $q > {z\choose 2}$ and $d\ge z - 1$. Then if $f$ is a random polynomial from $P_{d}$,
$$
\mathbb{P}[f(x_i) = 0\text{ for all } i = 1, \cdots, z] = \frac{1}{q^z}.
$$ 
\end{lemma}

We now define the graphs that we will be interested in. Let $N = q^k$. We will construct an $r$-partite $r$-uniform hypergraph on $rN$ vertices as follows. Let $V_1,\ldots, V_r$ be the partite sets, each a distinct copy of $\mathbb{F}_q^{k}$. Choose $f_1, f_2, \cdots f_{k(r-1)-1}:\mathbb{F}_q^{kr}\rightarrow \mathbb{F}_q$ to be independent random polynomials of degree $d := k(2k+1)$. For $v_1,\ldots, v_r$ with $v_i\in V_i$, we declare $\{v_1,\cdots, v_r\}$ to be an edge if and only if
$$
f_1({v_1}, {v_2}, \cdots ,{v_r}) = f_2({v_1}, {v_2}, \cdots ,{v_r}) = \cdots = f_{k(r-1)-1}({v_1},{v_2},\cdots ,{v_r}) = 0
$$

We use the term {\em random polynomial graph} to describe the distribution of hypergraphs obtained this way. Since these polynomials are chosen independently, we know from Lemma~\ref{Ind1} that the probability of a given hyperedge is in a random polynomial graph is $q^{1-k(r-1)}$. The total number of possible hyperedges is $N^{r} = q^{kr}$. Therefore the expected number of hyperedges is $q^{kr}q^{1-k(r-1)} = q^{k+1} = N^{1+1/k}$. 

We will be interested in subgraphs that appear in a random polynomial graph. Since hyperedges appear when a system of polynomials vanishes, we will describe subgraphs as varieties. Let $\overline{\mathbb{F}}_q$ be the algebraic closure of $\mathbb{F}_q$. A {\em variety} over $\overline{\mathbb{F}}_q$ is a set of the form:
\[
W = \{x\in \overline{\mathbb{F}}_q^{t}: f_1(x) = f_2(x) = \cdots = f_s(x) = 0\}
\]
For a collection of polynomials $f_1, \cdots, f_s:\overline{\mathbb{F}}_q^t\rightarrow \overline{\mathbb{F}}_q$. In other words, a variety is the set of common roots of a set of polynomials. We say $W$ is defined over $\mathbb{F}_q$ if the coefficients of these polynomials are from $\mathbb{F}_q$ and write $W(\mathbb{F}_q) = W\cap \mathbb{F}_q^t$. We say $W$ has complexity at most $M$ if $s$, $t$ and the maximum degree of the polynomials are all bounded by $M$. We will be very interested in how many points can be on a variety, and we will use the following theorem, proved by Bukh and Conlon (\cite{BukhConlon} Lemma 2.7) using tools from algebraic geometry.

\begin{theorem}\label{langweil}
Suppose $W$ and $D$ are varieties over $\overline{\mathbb{F}}_q$ of complexity at most $M$ which are defined over $\mathbb{F}_q$. Then one of the following holds:
\begin{itemize}
    \item $|W(\mathbb{F}_q) \setminus D(\mathbb{F}_q)| \leq c_M$, {where $c_M$ depends only on $M$, or}
    \item $|W(\mathbb{F}_q) \setminus D(\mathbb{F}_q)| \geq q/2$.
    \end{itemize}

\end{theorem}

Using these tools, we will show in the following section that we may modify a random polynomial graph to obtain a $\Theta_{k,t}^B$ free hypergraph with $\Omega(n^{1+1/k})$ edges with high probability.

\section{Proof of Theorem 4}\label{lower bound proof}
In this section we show that there exists an $n$ vertex $\Theta_{k,t}^B$ free hypergraph with $\Omega(n^{1+1/k})$ hyperedges. Our proof is an adaption of \cite{Conlon} to the hypergraph setting. Let $q$ be a sufficiently large prime power and let $G$ be a random polynomial graph on $n = rN = rq^k$ vertices defined in the previous section. As noted before the expected number of hyperedges in $G$ is 

\begin{equation}\label{number of edges}
N^{1+1/k} = \Omega(n^{1+1/k}).
\end{equation}

We are interested in $\Theta_{k,t}^B$ as a subgraph of $G$, and so we will be interested in Berge paths between vertices. Suppose now that $x$ and $y$ are two fixed vertices in $G$ and let $S$ be the set of Berge-paths with length $k$ between them. We will be interested in the moments of the random variable $|S|$. Let $m$ be fixed and note that $|S|^m$ is the number of collections of $m$ Berge paths of length $k$ between $x$ and $y$. These paths can be overlapping or identical, and the total number of hyperedges in any collection of $m$ paths is at most $km$. Since $q$ is sufficiently large, Lemma~\ref{Ind2} implies that for $z\leq d$ the probability of any particular collection with $z$ hyperedges is in $G$ is $q^{z(1-k(r-1))}$, since the probability that any hyperedge is in $G$ is $q^{1-k(r-1)}$. Now if we denote $P_{m,z}$ as the number of collections of $m$ paths between $x$ and $y$ such that their intersection has $z$ hyperedges in total, we have:
$$
\mathbb{E}[|S|^m] = \sum_{z = 1}^{km}P_{m,z}q^{z-zk(r-1)},
$$
as long as $km \leq d$. Now we shall estimate $P_{m,z}$ by estimating the maximum number of vertices there can be in any particular collection with $z$ hyperedges. Namely, if a collection has hyperedges $h_1, \cdots h_z$, then we want to estimate $\max|\cup_{i\in [z]}h_i|$. 

\begin{claim}\label{coreV}
If the union of $m$ Berge-paths from $x$ to $y$, each of length $k$, has $z$ edges in their intersection $\{h_i\}_{i=1}^z$, then 
\[
\left| \bigcup_{i=1}^z h_i \setminus\{x,y\}\right| \leq \frac{zk(r-1)-z}{k}.
\]

\end{claim}
\begin{proof}
Given such a collection, assume $|\cup h_i \setminus \{x,y\}| = n_0$. Let $P_1,\ldots, P_{m}$ be the set of paths, and let $n_i$ and $z_i$ be the number of vertices and edges respectively in $P_i \setminus \left(P_1 \cup \cdots \cup P_{i-1}\right)$. Let $z_i' = z_i$ if $n_i >0$ and $z_i' = 0$ if $n_i=0$.

If $n_i > 0$, then because consecutive edges in a Berge path must overlap, we have that $z_i \geq \lfloor\frac{n_i}{r-1}+1\rfloor$. Since $n_i$ is an integer, this implies that $z_i \geq \frac{n_i}{r-1} - \frac{r-2}{r-1} + 1 = \frac{n_i+1}{r-1}$. Let $r'$ be the number of $n_i$ which are greater than $0$. Let 
\[
z' = \sum_{i=1}^{m} z_i' \leq r'k,
\]
and so $r' \geq \frac{z'}{k}$. On the other hand,
\[
z' \geq \sum_{i: n_i>0} \frac{n_i + 1}{r-1} = \frac{n_0 + r'}{r-1} \geq \frac{n_0}{r-1} + \frac{z'}{k(r-1)}.
\]

This implies 
\[
n_0 \leq \frac{z'k(r-1)-z'}{k}.
\]
Since $z'\leq z$ the result follows.

\end{proof}

We can now bound $P_{m,z}$.
\begin{claim}
$P_{m,z} =  O_{k,r}(q^{zk(r-1)-z})$
\end{claim}
\begin{proof}
We want to count the number of all possible collections of $m$ Berge-paths between $x$ and $y$ with $z$ edges in their intersection. Let $V = \frac{zk(r-1)-z}{k}$ be the upper bound on the number of vertices, then we first choose the vertices that will be in the collection. There are less than $(rN)^V$ number of ways to do this. Then we choose $z$ hyperedges, each having $r$ vertices. The number such choices is bounded by ${V\choose r}^z$. Last we choose $2$ vertices from each hyperedge to make up the core vertices. This is bounded by ${r\choose 2}^z$, which eventually gives us 
$$
r^V{V\choose r}^z{r\choose 2}^zN^V =  O_{k,r}(N^V) = O_{k,r}(q^{zk(r-1)-z})
$$
Note that every collection of $m$ paths is counted by this method in at least one way, giving the upper bound on $P_{m,z}$.
\end{proof}

Thus, when $m\leq 2k+1$ we have $z\leq km \leq d$ (since $d = k(2k+1)$) and we may apply Lemma \ref{Ind2} to find  
\begin{equation}\label{moment}
\mathbb{E}[|S|^m] = \sum_{z = 1}^{km}P_{m,z}q^{z-zk(r-1)} \le kmC_{k,r} := C.
\end{equation}
Where $C_{k,r}$ is a constant dependent on $k$ and $r$, and $C$ is used to simplify our notation. 


Now if we want to apply tools from algebraic geometry, we need to write $S$ as a variety. However, this cannot be done directly since there is no fixed set of polynomials whose set of common roots is exactly $S$. Therefore, we use the following analysis to bound $|S|$. Any path in $S$ is a sequence of core vertices and edges $(x, h_1, v_1, h_2, \cdots, v_{k-1}, h_k, y)$. We may partition the set of paths into which partite set each $v_i$ is in. That is, $S$ can be partitioned into disjoint sets depending on which partite sets each core vertex belongs to. Namely, we can let $S_{t_1, \cdots t_{k-1}}$ denote the set of paths from $x$ to $y$ such that the $i$th core vertex $v_i$ belongs to $V_{t_i}$. 

Now if we let $\sigma$ denote any length $k-1$ tuple from $[r]^{k-1}$, then we have
$$
S = \bigcup_{\sigma\in [r]^{k-1}}S_\sigma,
$$
and this is a disjoint union.

Fix any arbitrary $S_\sigma$. For notation, we denote the core vertices in an arbitrary path as $v_1, \cdots, v_{k-1}$ and the non core vertices in hyperedge $h_i$ as $w_1^{i}, \cdots, w_{r-2}^i$. We also need to make sure that the non core vertices are ordered based on their partite sets. In other words, if $w_j^{i}\in V_{t_1}$ and $w_k^{i}\in V_{t_2}$ where $j < k$, then $t_1 < t_2$.

Now we may define the variety $T_\sigma$ as
\begin{ceqn}
\begin{align*}
\begin{split}& \{f_{i,1}(p) = \cdots = f_{i,k}(p) = 0\mbox{ for all $i$ in $[k(r-1)-1]$}\},   
\end{split}
\end{align*}    
\end{ceqn}

where $p\in \mathbb{F}_q^{k(r-2)+k-1}$ runs over sequences ($v_1,\cdots, v_{k-1}, w_1^1,\cdots, w_{r-2}^1, \cdots, w_1^k,\cdots, w_{r-2}^k$ (that is, each $p$ is a vector ordered with the core vertices first and the non core vertices after).

Here the polynomials $f_{i,1},\cdots, f_{i,k}$ are extensions to the polynomial $f_i$. Namely, 
$$f_{i,1}(p) = f_i(x, v_1, w_1^1, \cdots, w_{r-2}^1)$$
$$f_{i,2}(p) = f_i(v_1, v_2, w_1^2, \cdots, w_{r-2}^2)$$
$$\cdots$$
$$f_{i,k}(p) = f_i(v_{k-1}, y, w_1^k, \cdots, w_{r-2}^k)$$
Note that depending on $\sigma$ (the ordering of partite sets on core vertices), the arguments given to $f_i$ need to be reordered. However, since $\sigma$ is fixed and so is the ordering of all non core vertices, we can fix the order of arguments given to $f_i$ according to $\sigma$. For instance, if $x\in V_2$, $v_1\in V_3$, then 

$$
f_{i, 1}(p) = f_i(w_1^1, x, v_1, w_2^1, \cdots, w_{r-2}^1)
$$

With this restriction on ordering, we see that $S_\sigma\subseteq T_\sigma(\mathbb{F}_q)$. Note that $T_\sigma$ contains all of the paths in $S_\sigma$, but may also contain walks that are not paths.

If $T_\sigma(\mathbb{F}_q)$ contains a degenerate walk $x, v_1, v_2, \cdots, v_{k-1}, y$, then one of the following three conditions must be true: $x = v_b$ for some $b\in [k-1]$, $v_a = y$ for some $a\in [k-1]$, or $v_a = v_b$ for some $a\ne b\in [k-1]$. Therefore we can consider the collections of sets:
\begin{ceqn}
\begin{align*}
W_{0,b} &= T_\sigma \cap \{v_1, \cdots, w_{r-2}^k: x = v_b\} \\
W_{a,b} &= T_\sigma \cap \{v_1, \cdots, w_{r-2}^k: v_a = v_b\} \\
W_{a,0} &= T_\sigma \cap \{v_1, \cdots, w_{r-2}^k: v_a = y\}
\end{align*}
\end{ceqn}

Each of these sets is also a variety with complexity bounded in terms of $k$ and $r$. Let $W$ be the union of all of the $W_{0,b}$, $W_{a,b}$, and $W_{a,0}$. If $X$ and $Y$ are varieties with complexity bounded in terms of $k$ and $r$, then we claim that $X\cup Y$ is also a variety with complexity bounded in terms of $k$ and $r$. To see this, if $X$ is the set of points where the set of polynomials $\{f_i\}$ vanish and $Y$ is the set of points where the polynomials $\{g_j\}$ vanish, then $X\cup Y$ is exactly the set of points where the polynomials $\{f_ig_j\}$ vanish. Since $W$ is the union of $O(k^2)$ varieties, we have that $W$ is also a variety with complexity bounded in terms of $k$ and $r$. Since $S_\sigma = T_\sigma \setminus W$, we may apply Theorem \ref{langweil} to $S_\sigma$.



Now if we put everything together, we see that there exists a constant $c_{\sigma}$, dependent on $k$ and $r$, such that either $|S_\sigma|\le c_{\sigma}$ or $|S_\sigma|\ge \frac{q}{2}$. This conclusion holds true for any arbitrary $\sigma\in [r]^{k-1}$. Looking at $S$, we see that either $|S| = \sum_{\sigma\in [r]^{k=1}}|S_\sigma|\le C_{k,r}$ for some constant $C_{k,r}$ dependent on $k$ and $r$, or $|S| > C_{k,r}$, which implies that there exists $\sigma$ such that $|S_\sigma| > c_{\sigma}$ and therefore $|S|\ge |S_\sigma| > \frac{q}{2}$. Now by \eqref{moment} and Markov's inequality, for $m\leq 2k+1$ we have
$$
\mathbb{P}[|S| > C_{k,r}] = \mathbb{P}[|S| > \frac{q}{2}] =\mathbb{P}[|S|^m > (q/2)^m] \le \frac{C}{(q/2)^m}.
$$
Call a pair of vertices $(x,y)$ bad if it has more than $C_{k,r}$ length $k$ paths between them. If $B$ is the random variable denoting the number of bad pairs, then we have
\[
\mathbb{E}(B) \le N^2\times \frac{C}{(q/2)^m} = O_{k,r}(q^{2k-m}) = O_{k,r}\left(\frac{1}{q}\right),
\]
when we take $m=2k+1$. Therefore by Markov's Inequality, $\mathbb{P}[B \ge 1] \rightarrow 0$ as $n\rightarrow \infty$. Now let $X$ be the number of edges, then by \eqref{number of edges} the expected number of edges in is $N^{1+1/k}$. The variance is 
\[
\mathbb{E}[X^2] - \mathbb{E}[X]^2 = \sum_{i,j}\mathbb{E}[H_iH_j]-\mathbb{E}[H_i]\mathbb{E}[H_j] = N^r(q^{1-k(r-1)} - q^{2-2k(r-1)})\le \mathbb{E}[X]
\]
Where $H_i$ are indicator random variables for hyperedges and for $i\ne j$, $\mathbb{E}[H_iH_j] = \mathbb{E}[H_i]\mathbb{E}[H_j]$ because $H_i$ and $H_j$ are independent due to Lemma \ref{Ind2}. By Chebyshev's Inequality, 
\[
\mathbb{P}[|X - N^{1+1/k}|\ge \frac{1}{2}N^{1+1/k}] \le \frac{1}{4N^{1+1/k}}
\]
which goes to zero. Therefore with high probability this hypergraph on $rN$ vertices has $\Omega_{k,r}(N^{1+1/k})$ edges, and it contains no $\Theta_{k,C_{k,r}+1}^B$. This completes the proof of Theorem~\ref{lowerbound}.

\section{Conclusion}
In this paper we showed that for fixed $k,r,t$ there is a constant $c_{k,r,t}$ such that 
$$
\ex_r(n, \Theta_{k,t}^B) \le c_{r,k,t}n^{1+\frac{1}{k}},
$$
and that this order of magnitude is correct when $t$ is large enough relative to $r$ and $k$. That is, for fixed $k$ and $r$ there is a constant $c_{k,r}$ such that
$$
\ex_r(n, \Theta_{k,c_{k,r}}^B) = \Omega_{k,r}(n^{1+\frac{1}{k}}).
$$

We end with some open questions. First, it would be interesting to determine the dependence on $r$. Even when $t=2$ and $k\in \{2,3,5\}$ this dependence is unknown. For example, it is known that $\mathrm{ex}_r(n, C_4^B) = \Theta\left(n^{3/2}\right)$ when $2\leq r\leq 6$, but the order of magnitude is unknown for $r\geq 7$ (c.f. \cite{Tompkins}). It would also be interesting to determine the dependence on $t$ when $t$ is large. Finally, in this paper we worked with the least restrictive definition of paths between vertices in hypergraphs. One could forbid only certain types of paths and demand that no pair of vertices have more than $t$ of these paths between them. 
\bibliographystyle{plain}
\bibliography{bib}

\end{document}